\newif\ifpictures
\picturestrue

\newif\ifcomment
\commenttrue

\documentclass[11pt]{amsart}
\usepackage{amsmath}

\DeclareMathOperator*{\argmin}{argmin}
\usepackage{amsfonts}
\usepackage{amssymb}
\usepackage{latex_base}
\usepackage{macros}
\usepackage{times}
\usepackage{bm}
\usepackage{natbib}
\usepackage{graphicx}
\usepackage{subfigure}
\usepackage{tikz}
\usepackage{pgf-pie} % the package used to implement the pie charts
\usepackage{etex}
\usepackage{pgfplots}
\usepackage{diagbox}
\usepackage{multirow}
\usepackage{float}%强行指定图片位置
\usepackage{url}%插入链接

\makeatletter
\newif\if@restonecol
\makeatother

\usepackage[ruled,linesnumbered,vlined]{algorithm2e}

\author{Weiyi Ding}
\address{Weiyi Ding, School of Mathematical Sciences, Beihang University, Beijing, China}
\email{sy2009152@buaa.edu.cn}

\author{Xiaoxian Tang}
\address{Xiaoxian Tang, School of Mathematical Sciences, Beihang University, Beijing, China}
\email{xiaoxian@buaa.edu.cn}

\title[]{Projections of Tropical Fermat-Weber points}

\begin{document}
\title[]{Projections of Tropical Fermat-Weber points}

\begin{abstract}
In the tropical projective torus, it is not guaranteed that the projection of a Fermat-Weber point of a {given} data set is a Fermat-Weber point of the projection of the data set. In this paper, we focus on the projection on the tropical triangle (the three-point tropical convex hull), and {we} develop one algorithm (Algorithm \ref{algo:MofE}) and its improved version (Algorithm \ref{algo:NewMethod}), such that for a given data set {in the tropical projective torus}, these algorithms output a tropical triangle, on which the projection of a Fermat-Weber point of {the} data set is a Fermat-Weber point of the projection of {the} data set. We implement these algorithms in ${\tt R}$ and test how it works with random data sets. The experimental results show that, these algorithms can succeed with a much higher {probability} than choosing the tropical triangle randomly, the succeed rate of these two algorithms is stable while data sets are changing randomly, and Algorithm \ref{algo:NewMethod} can output the results much faster than Algorithm \ref{algo:MofE} averagely.
\end{abstract}

 \maketitle

\section{Introduction}
In this paper, we study {the} question: for a given data set $X$ in the tropical projective torus, how to find a tropical polytope ${\mathcal C}$, such that the projection of a Fermat-Weber point of $X$ on ${\mathcal C}$ is a Fermat-Weber point of the projection of $X$ on ${\mathcal C}$.

This problem is motivated by the tropical principal component analysis (tropical PCA) {proposed in} \cite{page2019tropical,yoshida2019tropical}, which is of great use in the analysis of phylogenetic trees in Phylogenetics. Phylogenetics is a subject that is very powerful for explaining genome evolution, processes of speciation and relationships among species. It offers a great challenge of analysing data sets that consist of phylogenetic trees.

Analysing data sets of phylogenetic trees with a fixed number of leaves is difficult because the space of phylogenetic trees is high dimensional and not Euclidean; it is a union of lower dimensional polyhedra cones in ${\mathbb R}^{\binom{n}{2}}$, where $n$ is the number of leaves \cite{page2019tropical}. Many multivariate statistical procedures have been applied to such data sets \cite{weyenberg2014kdetrees,gori2016clustering,hillis2005analysis,duchene2018analysis,knowles2018matter,yoshida2019multilocus}. People also have done a lot of work to apply PCA on data sets that consist of phylogenetic trees. For instance, Nye showed an algorithm \cite{nye2011principal} to compute the first order principal component over the space of phylogenetic trees. Nye \cite{nye2011principal} used a two-point convex hull under the CAT(0)-metric as the first order principal component over the Billera-Holmes-Vogtman (BHV) tree space {introduced in} \cite{billera2001geometry}. However, Lin et al. \cite{lin2017convexity} showed that the three-point convex hull {in the BHV tree space} can have arbitrarily high dimension, which means that the idea in \cite{nye2011principal} cannot be generalized to higher order principal components ({e.g., see} \cite{page2019tropical}). In addition, Nye et al. \cite{nye2017principal} used the locus of the weighted Fr\'echet mean when the weights vary over the $k$-simplex as the $k$-th principal component {in} the BHV tree space, and this approach performed well in simulation studies.

On the other hand, the tropical metric in tree spaces is well-studied \cite[{Chapter 5}]{maclagan2015introduction} and well-behaved \cite{lin2017convexity}. In 2019, Yoshida et al. \cite{yoshida2019tropical} defined the tropical PCA under the tropical metric in two ways: the Stiefel tropical linear space of fixed dimension, and the tropical polytope with a fixed number of vertices. Page et al. \cite{page2019tropical} used tropical polytopes for tropical PCA to visualize data sets of phylogenetic trees, and used Markov Chain Monte Carlo (MCMC) approach to optimally estimate the tropical PCA. Their experimental results \cite{page2019tropical} showed that, this {MCMC} method of computing tropical PCA performed well on both simulated data sets and empirical data sets.

This paper is motivated by a difference between classical PCA (in Euclidean spaces) and tropical PCA as follows. In classical PCA, the projection of the mean point of a data set $X$ in {the Euclidean space} is the mean point of the projection of $X$ (e.g., {see} \cite[Page 188]{zaki2014data}). However, in tropical PCA defined by tropical polytopes, the projection of a tropical mean point (in this paper we call it a Fermat-Weber point) of a data set $X$ is not necessarily a Fermat-Weber point of the projection of $X$ (see Example \ref{ex:NotKeepWithoutFWP}). {More specifically}, it is known that, for a data set $X$ {in the Euclidean space}, the mean point of $X$ is unique. However, for a data set $X$ in the tropical projective torus ({denoted by} ${\mathbb R}^{n}\!/{\mathbb R}{\mathbf 1}$), the Fermat-Weber point of $X$ is not necessarily unique \cite[Proposition 20]{yoshida2020tropical}. For a data set $X\subset{\mathbb R}^{n}\!/{\mathbb R}{\mathbf 1}$ and a tropical convex hull ${\mathcal C}$, the tropical projection \cite[Formula 3.3]{kang2019unsupervised} of the set of Fermat-Weber points of $X$ on ${\mathcal C}$ are not exactly equal to the set of Fermat-Weber points of the projection of $X$ on ${\mathcal C}$. In addition, it is also known that, in ${\mathbb R}^{n}\!/{\mathbb R}{\mathbf 1}$, if a set is {the} union of $X$ and a Fermat-Weber point of $X$, then {the} union has exactly one Fermat-Weber point \cite[Lemma 8]{lin2018tropical}. So a natural question is, if a set is {the} union of $X$ and a Fermat-Weber point of $X$, can the projection of the Fermat-Weber point of the union be a Fermat-Weber point of the projection of {the} union? By experiments we know that this is {still} not guaranteed, and {it} depends on the choice of the tropical convex hull ${\mathcal C}$ (see Example \ref{ex:NotKeepWithFWP}).

In this paper, we focus on tropical triangles (three-point tropical polytopes). We develop one algorithm (Algorithm \ref{algo:MofE}) and its improved version (Algorithm \ref{algo:NewMethod}), such that for a given data set $X\subset{\mathbb R}^{n}\!/{\mathbb R}{\mathbf 1}$, these algorithms output a tropical triangle ${\mathcal C}$, on which the projection of a Fermat-Weber point of $X$ is a Fermat-Weber point of the projection of $X$. By sufficient experiments with random data sets, we show that Algorithm \ref{algo:MofE} and Algorithm \ref{algo:NewMethod} can both succeed with a much higher {probability} than choosing a tropical triangle ${\mathcal C}$ randomly (see Table \ref{tab:showMyAlgoIsGood} and Table \ref{tab:goodRateForMNchange}). We also show that the succeed rate of these two algorithms is stable while data sets are changing randomly (see Table \ref{tab:GoodRateVchanges}). Algorithm \ref{algo:NewMethod} can output the result much faster than Algorithm \ref{algo:MofE} does averagely (see Table \ref{tab:runTimeForBoth}), {because in most cases, Algorithm \ref{algo:NewMethod} correctly terminates with less steps than Algorithm \ref{algo:MofE} does (see Figure \ref{fig:VarChangingCompare})}.

This paper is organized as follows. In Section \ref{sec:basic}, we remind readers of the basic definitions in tropical geometry. In Section \ref{sec:projection}, we prove Theorem \ref{theo:algorithmIsRight} and Theorem \ref{lem:cannotBothHold} for the correctness of the algorithms developed in this paper. In Section \ref{sec:algorithm}, we {present Algorithm \ref{algo:MofE} and Algorithm \ref{algo:NewMethod}. We also explain how the algorithms work by two examples}. In Section \ref{sec:experiment}, we apply the algorithms developed in Section \ref{sec:algorithm} on random data sets, and illustrate the experimental results.

\section{Tropical Basics}\label{sec:basic}
In this section, {we set up the notation throughout this paper, and introduce some basic tropical arithmetic and geometry}.
\begin{definition}[\bf Tropical Arithmetic Operations]
We denote by $(\mathbb{R}\cup\{-\infty\},\boxplus,\odot)$ the {\em max-plus tropical semi-ring}. We define the {\em tropical addition} and the {\em tropical multiplication} as :
\begin{equation*}
c\boxplus{d}:=\max\{c,d\},\;\;\;c\odot{d}:=c+d\text{, \;\;where }c, d\in\mathbb{R}\cup\{-\infty\}.
\end{equation*}
\end{definition}

\begin{definition}[\bf Tropical Vector Addition]
For any scalars $c,d\in\mathbb{R}\cup\{-\infty\}$, and for any vectors $${\bf u}=(u_1,\dots,u_n),\;{\bf v}=(v_1,\dots,v_n)\in(\mathbb{R}\cup\{-\infty\})^n,$$ we define the {\em tropical vector addition} as:
\begin{align*}
c\odot {\bf u}\boxplus d\odot {\bf v}&:=(\max\{c+{u}_1,d+{v}_1\},\dots,\max\{c+{u}_n,d+{v}_n\}).
\end{align*}
\end{definition}

\begin{example}
\label{ex:tropicalLinearCombi}
Let
\begin{equation*}
{\bf u}=(2,1,3),\;\;{\bf v}=(2,2,2).
\end{equation*}
Also we let $c=-2,d=1$. Then we have
\begin{equation*}
c\odot {\bf u}\boxplus d\odot {\bf v}=(\max\{-2+2,1+2\},\max\{-2+1,1+2\},\max\{-2+3,1+2\})=(3,3,3).
\end{equation*}
\end{example}

\par For any point ${\bf u}\in{\mathbb R}^n$, we define {\em the equivalence class} $[{\bf u}]:=\{{\bf u}+c\cdot\mathbf{1}|c\in\mathbb{R}\},\;\text{where}\;{\mathbf 1}=(1,\dots,1).$ For instance, the vector $(3,3,3)$ is equivalent to (0,0,0). In the rest of this paper, we consider the {\em tropical projective torus}
\begin{equation*}
\mathbb{R}^n\!/\mathbb{R}\mathbf{1}:=\{[{\bf u}]|{\bf u}\in{\mathbb R}^n\}.
\end{equation*}
For convenience, we simply denote by ${\bf u}$ its equivalence class instead of $[{\bf u}]$, {and} we assume the first coordinate of every point in $\mathbb{R}^n\!/\mathbb{R}\mathbf{1}$ is $0$. Because for any ${\bf u}=(u_1,\dots,u_n)\in\mathbb{R}^n\!/\mathbb{R}\mathbf{1}$, it is equivalent to
\begin{equation}
\label{equ:regulateToZero}
{\bf u}=(0,u_2-u_1,\dots,u_n-u_1).
\end{equation}

\begin{definition}[\bf Tropical Distance]
For any two points
{$${\bf u}=(u_1,\dots,u_n),\;\;{\bf v}=(v_1,\dots,v_n)\in\mathbb{R}^n\!/\mathbb{R}\mathbf{1},$$}
we define the {\em tropical distance} $d_{tr}({\bf u},{\bf v})$ as:
\begin{equation*}
d_{tr}({\bf u},{\bf v}):=\max\{|u_i-v_i-u_j+v_j|:1\leq i<j\leq n\}=\max\limits_{{1\leq i\leq n}}\{u_i-v_i\}-\min\limits_{{1\leq i\leq n}}\{u_i-v_i\}.
\end{equation*}
\end{definition}

Note that the tropical distance is a metric in $\mathbb{R}^n\!/\mathbb{R}\mathbf{1}$ \cite[Page 2030]{lin2017convexity}.

\begin{example}
Let ${\bf u}=(0,4,2),{\bf v}=(0,1,1)\in\mathbb{R}^3\!/\mathbb{R}\mathbf{1}$. The tropical distance between ${\bf u}$, ${\bf v}$ is
\begin{equation*}
d_{tr}({\bf u},{\bf v})=\max\{0,3,1\}-\min\{0,3,1\}=3-0=3.
\end{equation*}
\end{example}

\begin{definition}[\bf Tropical Convex Hull]
Given a finite subset {$$X=\{{\bf x}^{(1)},\dots,{\bf x}^{(t)}\}\subset\mathbb{R}^n\!/\mathbb{R}\mathbf{1},$$} we define the {\em tropical convex hull}  as {the set of all tropical linear combinations of $X$:}
\begin{equation*}
tconv(X):=\{c_1\odot {\bf x}^{(1)}\boxplus c_2\odot {\bf x}^{(2)}\boxplus\dots\boxplus c_t\odot {\bf x}^{(t)}|c_1,\dots,c_t\in\mathbb{R}\}.
\end{equation*}
If $|X|=3,$ then the tropical convex hull of $X$ is called a {\em tropical triangle}.
\end{definition}

\begin{example}
Consider a set $X=\{{\bf x}^{(1)},{\bf x}^{(2)},{\bf x}^{(3)}\}\subset\mathbb{R}^3\!/\mathbb{R}\mathbf{1}$, where
\begin{equation*}
{\bf x}^{(1)}=(0,0,0),\;\;{\bf x}^{(2)}=(0,4,2),\;\;{\bf x}^{(3)}=(0,2,4).
\end{equation*}
The tropical convex hull $tconv(X)$ is shown in Figure \ref{fig:tconv examp}. Note that $\mathbb{R}^3\!/\mathbb{R}\mathbf{1}$ is isomorphic to $\mathbb{R}^2$ \cite{speyer2004tropical}, so the points in Figure \ref{fig:tconv examp} are drawn on a plane.
\end{example}

\begin{figure}[ht]
\centering
\begin{tikzpicture}[scale=0.8]
\draw[->](-1,0)--(5,0);
\draw[->](0,-1)--(0,5);
\draw[blue,very thick,fill=blue](0,0)--(0,2)--(2,4)--(4,4)--(4,2)--(2,0)--(0,0);
\draw[red,fill](0,0) circle [radius=0.05];
\draw[red,fill](2,4) circle [radius=0.05];
\draw[red,fill](4,2) circle [radius=0.05];
\node at (-0.6,-0.3) {(0,0,0)};
\node at (-0.6,2) {(0,0,2)};
\node at (1.4,4.2) {(0,2,4)};
\node at (4.6,4.2) {(0,4,4)};
\node at (4.8,2) {(0,4,2)};
\node at (2,-0.3) {(0,2,0)};
draw[red,fill=red](0,0) circle [radius=0.5];
\end{tikzpicture}
\caption{Blue region is the tropical convex hull of {the set of red points}}
\label{fig:tconv examp}
\end{figure}
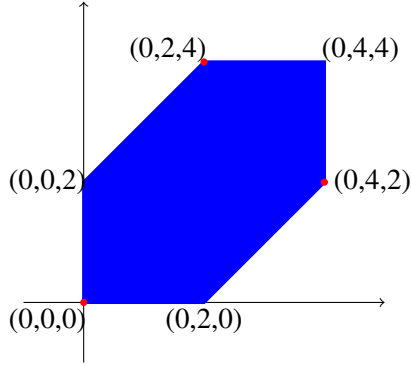

\begin{definition}[\bf Tropical Fermat-Weber Points]
\label{def:FWP}
Suppose we have
$$X=\{{\bf x}^{(1)},\dots,{\bf x}^{(t)}\}\subset\mathbb{R}^n\!/\mathbb{R}\mathbf{1}.$$
We define the set of {\em tropical Fermat-Weber points} of $X$ as
\begin{equation}
\label{equ:defOfFWP}
\mathop{\argmin}_{{\bf y}\in\mathbb{R}^n\!/\mathbb{R}\mathbf{1}}\sum_{i=1}^{t}{d_{tr}({\bf y},{\bf x}^{(i)})}.
\end{equation}
{The Fermat-Weber point of $X$ is denoted by $F_X$.}
\end{definition}

\begin{proposition}\cite[Proposition 25]{lin2017convexity}
\label{prop:FWpointLP}
Given $X=\{{\bf x}^{(1)},\dots,{\bf x}^{(t)}\}\subset\mathbb{R}^n\!/\mathbb{R}\mathbf{1}$, the set of tropical Fermat-Weber points of $X$ in $\mathbb{R}^n\!/\mathbb{R}\mathbf{1}$ is a convex polytope in $\mathbb{R}^{n-1}$. It consists of all optimal solutions ${\bf y}=(y_1,\dots,y_n)$ to the linear programming {problem:}
\begin{equation}\label{equ:FWpointLP}
\begin{split}
\text{mini}&\text{mize}\sum_{i=1}^{t}{\gamma_i},\\
\text{subject to}\;\gamma_i&\geq y_k-x_k^{(i)}-y_\ell+x_\ell^{(i)},\\
\gamma_i&\geq-(y_k-x_k^{(i)}-y_\ell+x_\ell^{(i)}), \\
\text{for all}\;1\leq k<\ell&\leq n,\;\text{and for all}\;i\in\{1,2,\dots,t\}.
\end{split}
\end{equation}
\end{proposition}

\begin{definition}[\bf Tropical Projection]
\label{def:tropicalProj}
Let $$U=\{{\bf u}^{(1)}=(u^{(1)}_1,\dots,u^{(1)}_n),\dots,{\bf u}^{(t)}=(u^{(t)}_1,\dots,u^{(t)}_n)\}\subset\mathbb{R}^{n}\!/\mathbb{R}\mathbf{1}.$$ Also let $\mathcal{C}=tconv(U)$. For any point ${\bf x}=(x_1,\dots,x_n)\in\mathbb{R}^{n}\!/\mathbb{R}\mathbf{1}$, we define the {\em projection of ${\bf x}$} on ${\mathcal C}$ as:
\begin{equation}
\label{equ:calculProjection}
\delta_{\mathcal{C}}({\bf x}):=\lambda_1\odot {\bf u}^{(1)}\boxplus\lambda_2\odot {\bf u}^{(2)}\boxplus\dots\boxplus\lambda_t\odot {\bf u}^{(t)},
\end{equation}
where $\lambda_i:=\min\{x_1-u_1^{(i)},\dots,x_n-u_n^{(i)}\}$ for all $i\in\{1,\dots,t\}$ \cite[Formula 3.3]{kang2019unsupervised}.
\end{definition}

\begin{proposition}\cite[Lemma 8]{lin2018tropical}
\label{prop:LinBo}
Let $X=\{{\bf x}^{(1)},\dots,{\bf x}^{(m)}\}\subset\mathbb{R}^{n}\!/\mathbb{R}\mathbf{1}$. Suppose $F_X$ is a Fermat-Weber point of $X$. Then $\{F_X,{\bf x}^{(1)},\dots,{\bf x}^{(m)}\}$ has exactly one Fermat-Weber point, which is $F_X$.
\end{proposition}

\subsection{Examples}
\begin{example}
\label{ex:NotKeepWithoutFWP}
This example shows that, for a given data set $X\subset\mathbb{R}^{3}\!/\mathbb{R}\mathbf{1}$ and a given two-point tropical polytope ${\mathcal C}$, the projection of a Fermat-Weber point of $X$ on ${\mathcal C}$ is not necessarily a Fermat-Weber point of the projection of $X$ on ${\mathcal C}$.

Suppose we have $X=\{(0,1,5),(0,2,4),(0,3,1),(0,4,3)\}\subset\mathbb{R}^{3}\!/\mathbb{R}\mathbf{1}.$ By solving the linear programming \eqref{equ:FWpointLP} in Proposition \ref{prop:FWpointLP} (e.g., using {\tt lpSolve} in {\tt R}), {we obtain that, $(0,3,3)$ is a Fermat-Weber point of $X$}. Let ${\mathcal C}=tconv(\{(0,0,2),(0,\frac{33}{10},2)\}).$ Then the projection of $X$ on ${\mathcal C}$ is $P=\{(0,1,2),(0,2,2),(0,\frac{33}{10},2)\}.$

We remark that, in $P$, {$(0,1,2)$ is the projection of $(0,1,5)$, $(0,2,2)$ is the projection of $(0,2,4)$, and $(0,\frac{33}{10},2)$ is the projection of both $(0,3,1)$ and $(0,4,3)$ on ${\mathcal C}$.}

{Note that $(0,2,2)$ is the unique Fermat-Weber point of $P$, while the projection of a Fermat-Weber point $(0,3,3)$ of $X$ is $(0,3,2)$.} So we can see that {the projection of a Fermat-Weber point of $X$ on ${\mathcal C}$ is not a Fermat-Weber point of the projection} (see Figure \ref{fig:NotKeepWithoutFWP}).
\end{example}

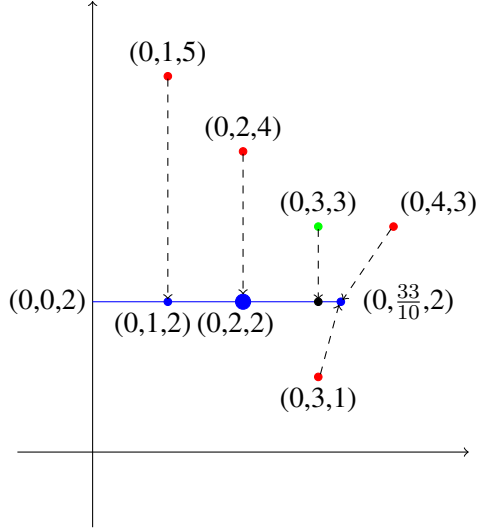
\begin{figure}[H]
\centering
\begin{tikzpicture}[]
\draw[->](-1,0)--(5,0);
\draw[->](0,-1)--(0,6);
\draw[red,fill](1,5) circle [radius=0.05];
\draw[red,fill](2,4) circle [radius=0.05];
\draw[red,fill](3,1) circle [radius=0.05];
\draw[red,fill](4,3) circle [radius=0.05];
\draw[green,fill](3,3) circle [radius=0.05];
\draw[blue](0,2)--(3.3,2);
\draw[blue,fill](1,2) circle [radius=0.05];
\draw[blue,fill](2,2) circle [radius=0.1];
\draw[black,fill](3,2) circle [radius=0.05];
\draw[blue,fill](3.3,2) circle [radius=0.05];
\draw[dashed,->](1,4.968)--(1,2.035);
\draw[dashed,->](2,3.968)--(2,2.1);
\draw[dashed,->](3,2.968)--(3,2.032);
\draw[dashed,->](3.028,1.035)--(3.275,1.96);
\draw[dashed,->](3.968,2.968)--(3.325,2.025);
%\node at (-0.6,-0.3) {(0,0,0)};
\node at (-0.6,2) {(0,0,2)};
\node at (1,5.3) {(0,1,5)};
\node at (2,4.3) {(0,2,4)};
\node at (3,3.3) {(0,3,3)};
\node at (4.6,3.3) {(0,4,3)};
\node at (3,0.7) {(0,3,1)};
\node at (4.2,2) {(0,$\frac{33}{10}$,2)};
\node at (1.9,1.7) {(0,2,2)};
\node at (0.8,1.7) {(0,1,2)};
draw[red,fill=red](0,0) circle [radius=0.5];
\end{tikzpicture}
\caption{{\bf1.} Red points are points in $X$. The green point is a Fermat-Weber point of $X$.\protect\\
{\bf2.} The blue line segment is the tropical convex hull ${\mathcal C}$ generated by $(0,0,2)$ and $(0,\frac{33}{10},2)$.\protect\\
{\bf3.} Blue points are the projection $P$ of $X$. And the biggest blue one $(0,2,2)$ is the Fermat-Weber point of $P$. The black point $(0,3,2)$ is the projection of the green point.\protect\\
}
\label{fig:NotKeepWithoutFWP}
\end{figure}

\begin{example}
\label{ex:NotKeepWithFWP}
This example shows that, in $\mathbb{R}^{n}\!/\mathbb{R}\mathbf{1}$, if a set ${\widetilde X}$ is {the} union of $X$ and a Fermat-Weber point $F_X$ of $X$, then it is not guaranteed that the projection of the Fermat-Weber point $F_X$ of ${\widetilde X}$ is a Fermat-Weber point of the projection of ${\widetilde X}$. Besides, whether the projection of the Fermat-Weber point $F_X$ of ${\widetilde X}$ is a Fermat-Weber point of the projection of ${\widetilde X}$ depends on the choice of the tropical convex hull ${\mathcal C}$.

Suppose we have $$X=\{(0,1,5),(0,2,4),(0,3,1),(0,4,3)\}\subset\mathbb{R}^{3}\!/\mathbb{R}\mathbf{1}.$$ By solving the linear programming \eqref{equ:FWpointLP} in Proposition \ref{prop:FWpointLP}, we obtain that, $(0,3,3)$ is a Fermat-Weber point of $X$. Then $${\widetilde X}=X\cup \{F_X\}=\{(0,1,5),(0,2,4),(0,3,1),(0,4,3),(0,3,3)\}.$$

Let ${\mathcal C}_1=tconv(\{(0,0,2),(0,\frac{5}{2},2)\}),\;\;{\mathcal C}_2=tconv(\{(0,0,2),(0,4,2)\}).$ $P_1$ and $P_2$ are the projection of ${\widetilde X}$ on ${\mathcal C}_1$ and ${\mathcal C}_2$ respectively, where $$P_1=\{(0,1,2),(0,2,2),(0,\frac{5}{2},2)\},\;\;P_2=\{(0,1,2),(0,2,2),(0,3,2),(0,4,2)\}.$$

{We remark that, in $P_1$, $(0,1,2)$ is the projection of $(0,1,5)$, $(0,2,2)$ is the projection of $(0,2,4)$, and $(0,\frac{5}{2},2)$ is the projection of $(0,3,1)$, $(0,4,3)$ and $(0,3,3)$ on ${\mathcal C}_1$. And in $P_2$, $(0,1,2)$ is the projection of $(0,1,5)$, $(0,2,2)$ is the projection of $(0,2,4)$, $(0,3,2)$ is the projection of $(0,3,3)$, and $(0,4,2)$ is the projection of $(0,3,1)$ and $(0,4,3)$ on ${\mathcal C}_2$.}

{Note that $(0,2,2)$ is the unique Fermat-Weber point of $P_1$, while the projection of the Fermat-Weber point $(0,3,3)$ of ${\widetilde X}$ on ${\mathcal C}_1$ is $(0,\frac{5}{2},2)$. So we can see that the projection of the Fermat-Weber point of ${\widetilde X}$ on ${\mathcal C}_1$ is not a Fermat-Weber point of the projection.} On the other hand, the projection of the Fermat-Weber point $(0,3,3)$ on ${\mathcal C}_2$ is $(0,3,2)$, which is {exactly} a Fermat-Weber point of the projection (see Figure \ref{fig:NotKeepWithFWP}).
\end{example}

\begin{figure}[ht]
\centering
\subfigure[]{
\begin{tikzpicture}[scale=0.9]
\draw[->](-1,0)--(5,0);
\draw[->](0,-1)--(0,6);
\draw[red,fill](1,5) circle [radius=0.05];
\draw[red,fill](2,4) circle [radius=0.05];
\draw[red,fill](3,1) circle [radius=0.05];
\draw[red,fill](4,3) circle [radius=0.05];
\draw[red,fill](3,3) circle [radius=0.05];
\draw[blue](0,2)--(2.5,2);
\draw[blue,fill](1,2) circle [radius=0.05];
\draw[blue,fill](2,2) circle [radius=0.05];
\draw[blue,fill](2.5,2) circle [radius=0.05];

\draw[dashed,->](1,4.968)--(1,2.035);
\draw[dashed,->](2,3.968)--(2,2.035);
\draw[dashed,->](3,2.968)--(2.525,2.032);
\draw[dashed,->](3.018,1.035)--(2.525,1.96);
\draw[dashed,->](3.968,2.968)--(2.528,2.025);
%\node at (-0.6,-0.3) {(0,0,0)};
\node at (-0.6,2) {(0,0,2)};
\node at (1,5.3) {(0,1,5)};
\node at (2,4.3) {(0,2,4)};
\node at (3,3.3) {(0,3,3)};
\node at (4.6,3.3) {(0,4,3)};
\node at (3,0.7) {(0,3,1)};
\node at (1.9,1.7) {(0,2,2)};
\node at (0.8,1.7) {(0,1,2)};
\node at (3.3,1.95) {(0,$\frac{5}{2}$,2)};
draw[red,fill=red](0,0) circle [radius=0.5];
\end{tikzpicture}}
\subfigure[]{
\begin{tikzpicture}[scale=0.9]
\draw[->](-1,0)--(5,0);
\draw[->](0,-1)--(0,6);
\draw[red,fill](1,5) circle [radius=0.05];
\draw[red,fill](2,4) circle [radius=0.05];
\draw[red,fill](3,1) circle [radius=0.05];
\draw[red,fill](4,3) circle [radius=0.05];
\draw[red,fill](3,3) circle [radius=0.05];
\draw[blue](0,2)--(4,2);
\draw[blue,fill](1,2) circle [radius=0.05];
\draw[blue,fill](2,2) circle [radius=0.05];
\draw[blue,fill](3,2) circle [radius=0.05];
\draw[blue,fill](4,2) circle [radius=0.05];
\draw[dashed,->](1,4.968)--(1,2.035);
\draw[dashed,->](2,3.968)--(2,2.035);
\draw[dashed,->](3,2.968)--(3,2.032);
\draw[dashed,->](3.028,1.035)--(3.975,1.96);
\draw[dashed,->](4,2.968)--(4,2.025);
%\node at (-0.6,-0.3) {(0,0,0)};
\node at (-0.6,2) {(0,0,2)};
\node at (1,5.3) {(0,1,5)};
\node at (2,4.3) {(0,2,4)};
\node at (3,3.3) {(0,3,3)};
\node at (4.6,3.3) {(0,4,3)};
\node at (3,0.7) {(0,3,1)};
\node at (4.7,2) {(0,4,2)};
\node at (1.9,1.7) {(0,2,2)};
\node at (0.8,1.7) {(0,1,2)};
\node at (3,1.7) {(0,3,2)};
draw[red,fill=red](0,0) circle [radius=0.5];
\end{tikzpicture}
}
\caption{{\bf1.} {Points in ${\widetilde X}$ are red. The projection points of ${\widetilde X}$ are blue.}\protect\\
{\bf2.} {In (a), the blue line segment is the tropical convex hull generated by $(0,0,2)$ and $(0,\frac{5}{2},2)$.}\protect\\
{\bf3.} {In (a), blue points are the projection $P_1$ of ${\widetilde X}$. Note that $(0,\frac{5}{2},2)$ is the projection of the Fermat-Weber point $(0,3,3)$ of ${\widetilde X}$, and $(0,2,2)$ is the unique Fermat-Weber point of $P_1$.}\protect\\
{\bf4.} {In (b), the blue line segment is the tropical convex hull generated by $(0,0,2)$ and $(0,4,2)$.}\protect\\
{\bf5.} In (b), blue points are the projection $P_2$ of ${\widetilde X}$. Note that $(0,3,2)$ is the projection of the Fermat-Weber point $(0,3,3)$ of ${\widetilde X}$, which is a Fermat-Weber point of $P_2$.}
\label{fig:NotKeepWithFWP}
\end{figure}
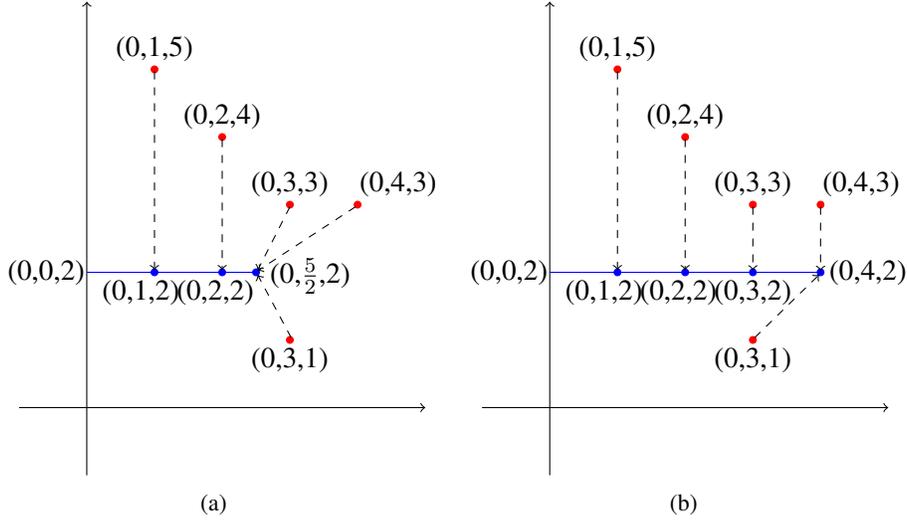

\section{Theorems}\label{sec:projection}
In this section, {we introduce Theorem \ref{theo:algorithmIsRight} and Theorem \ref{lem:cannotBothHold} for proving the correctness of the algorithms developed in the next section.}

\begin{lemma}
\label{lem:projToTriangleBelow(manyPoints)}
Suppose we have a data set $$X=\{{\bf x}^{(1)}=(x^{(1)}_{1},x^{(1)}_{2},\dots,x^{(1)}_{n}),\dots,{\bf x}^{(m)}=(x^{(m)}_{1},x^{(m)}_{2},\dots,x^{(m)}_{n})\}\subset\mathbb{R}^{n}\!/\mathbb{R}\mathbf{1}.$$Let $t$ be a number which is no more than
$$\min\limits_{1\leq k\leq m}\min\limits_{1\leq \ell\leq n}\{x^{(k)}_\ell\}.$$
For any two fixed integers $d_1\;\text{and}\;d_2\;(2\leq d_1<d_2\leq n)$, we define three points ${\bf u}^{(1)},{\bf u}^{(2)},{\bf u}^{(3)}\in\mathbb{R}^{n}\!/\mathbb{R}\mathbf{1}$ as follows.

\begin{align}
&\text{for}\;k=1,2,3,&u^{(k)}_1&:=0,\label{Equ:0}\\
&u^{(1)}_{d_1}:=\min\limits_{1\leq k\leq m}\{x^{(k)}_{d_1}\}-1,&u^{(1)}_{d_2}&:=\min\limits_{1\leq k\leq m}\{x^{(k)}_{d_2}\}-1,\label{Equ:1}\\
&u^{(2)}_{d_1}:=\min\limits_{1\leq k\leq m}\{x^{(k)}_{d_1}\}+1,&u^{(2)}_{d_2}&:=\max\limits_{1\leq k\leq m}\{x^{(k)}_{d_2}\}+1,\label{Equ:2}\\
&u^{(3)}_{d_1}:=\max\limits_{1\leq k\leq m}\{x^{(k)}_{d_1}\}+1,&u^{(3)}_{d_2}&:=\min\limits_{1\leq k\leq m}\{x^{(k)}_{d_2}\}+1,\label{Equ:3}\\
&\text{for}\;k=1,2,3,\;\text{and for all}\;\ell\neq1,d_1,d_2,&u^{(k)}_\ell&:=t.\label{Equ:-1}
\end{align}
Let ${\mathcal C}=tconv(\{{\bf u}^{(1)},{\bf u}^{(2)},{\bf u}^{(3)}\})$. Then, the projection of $X$ on ${\mathcal C}$ is
\begin{equation}
\label{equ:projectionOfX}
\delta_{\mathcal C}({\bf x}^{(k)})=(0,t,\dots,t,x^{(k)}_{d_1},t,\dots,t,x^{(k)}_{d_2},t,\dots,t),\;\text{for all}\;k\in\{1,2,\dots,m\},
\end{equation}
{where $x^{(k)}_{d_1}$ and $x^{(k)}_{d_2}$ are respectively located at the $d_1$-th and $d_2$-th coordinates of $\delta_{\mathcal C}({\bf x}^{(k)})$.}
\end{lemma}

\begin{proof}
Recall that we assume the first coordinate of every point in $\mathbb{R}^{n}\!/\mathbb{R}\mathbf{1}$ is $0$. For any
$${\bf x}^{(i)}=(x^{(i)}_{1},x^{(i)}_{2},\dots,x^{(i)}_{n})\in X,$$
{by Definition \ref{def:tropicalProj},} we have {that $\lambda_i$ in \eqref{equ:calculProjection} should be:}
$$\lambda_1=0,\;\;\lambda_2=x^{(i)}_{d_2}-\max\limits_{1\leq k\leq m}\{x^{(k)}_{d_2}\}-1,\;\;\lambda_3=x^{(i)}_{d_1}-\max\limits_{1\leq k\leq m}\{x^{(k)}_{d_1}\}-1.$$
Then the conclusion follows from \eqref{equ:calculProjection}.
\end{proof}

{Suppose $X$ is the data set stated in Lemma \ref{lem:projToTriangleBelow(manyPoints)}. For ${\bf u}^{(1)}$, ${\bf u}^{(2)}$ and ${\bf u}^{(3)}$ in Lemma \ref{lem:projToTriangleBelow(manyPoints)}, let ${\mathcal C}=tconv(\{{\bf u}^{(1)},{\bf u}^{(2)},{\bf u}^{(3)}\})$, we have the following remarks:} the equalities \eqref{Equ:1}-\eqref{Equ:3} make sure that the tropical triangle ${\mathcal C}$ is big enough; the equalities \eqref{Equ:0} and \eqref{Equ:-1} make sure that ${\mathcal C}$ parallels with a coordinate plane; the equality \eqref{Equ:-1} makes sure that ${\mathcal C}$ is located under all points in $X$. Lemma \ref{lem:projToTriangleBelow(manyPoints)} shows that we can project $X$ vertically onto ${\mathcal C}$ (see Example \ref{ex:projectionForManyPoints} and Figure \ref{fig:projectionExamp}).

\begin{figure}[H]
\centering
\begin{tikzpicture}[scale=0.8]
\begin{axis}[]
\addplot3[only marks,fill=red,]
coordinates{(2,1,3) (1,1,4) (3,2,3) (3,3,5) (2,3,2)};%数据点
\addplot3[fill=green,draw=black]
coordinates{(2,0,-1) (0,0,-1) (0,2,-1) (2,4,-1) (4,4,-1) (4,2,-1) (2,0,-1)};%三角形
\addplot3[only marks,mark size=2pt,fill=blue]
coordinates{(2,1,-1) (1,1,-1) (3,2,-1) (3,3,-1) (2,3,-1)};
\addplot3[dashed,->]
coordinates{(2,1,3) (2,1,-0.87)};
\addplot3[dashed,->]
coordinates{(1,1,4) (1,1,-0.87)};
\addplot3[dashed,->]
coordinates{(3,2,3) (3,2,-0.87)};
\addplot3[dashed,->]
coordinates{(3,3,5) (3,3,-0.87)};
\addplot3[dashed,->]
coordinates{(2,3,2) (2,3,-0.87)};
%\node at (4,3,3) {(0,2,1,3)};
\end{axis}
\end{tikzpicture}
\caption{How data points (red) project onto {the tropical triangle} (green) in Lemma \ref{lem:projToTriangleBelow(manyPoints)}}
\label{fig:projectionExamp}
\end{figure}
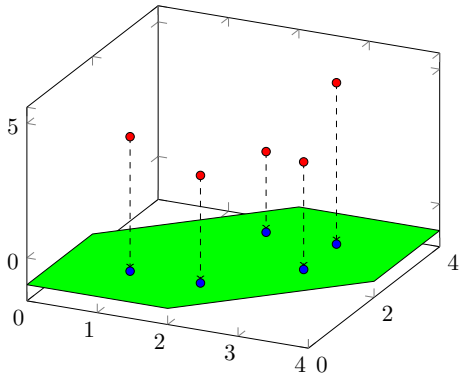

\begin{example}
\label{ex:projectionForManyPoints}
Suppose we have $$X=\{(0,2,1,3),(0,1,1,4),(0,3,2,3),(0,3,3,5),(0,2,3,2)\}\subset\mathbb{R}^{4}\!/\mathbb{R}\mathbf{1}.$$ Let $t=-1$. Fix $d_1=2,\;\text{and}\;d_2=3$. By \eqref{Equ:0}-\eqref{Equ:-1}, we can define three points ${\bf u}^{(1)},{\bf u}^{(2)},\;\text{and}\;{\bf u}^{(3)}$ as:
$${\bf u^{(1)}}=(0,0,0,-1),\;\;{\bf u^{(2)}}=(0,2,4,-1),\;\;{\bf u^{(3)}}=(0,4,2,-1).$$
Let ${\mathcal C}=tconv(\{{\bf u^{(1)}},{\bf u^{(2)}},{\bf u^{(3)}}\})$ ({see the green region in Figure \ref{fig:projectionExamp}}). Then, by Lemma \ref{lem:projToTriangleBelow(manyPoints)}, the projection points of $X$ on ${\mathcal C}$ are shown in Figure \ref{fig:projectionExamp} ({see the blue points}).
\end{example}

\begin{definition}[{\bf Data Matrix}]
We define any matrix $X$ with {$n$ columns} as a {\em data matrix}, where each row of $X$ is regarded as a point in $\mathbb{R}^{n}\!/\mathbb{R}\mathbf{1}$.
\end{definition}

{Below we denote by $X_{m\times n}$ the data matrix $X$ with size $m\times n$.}

\begin{definition}[{\bf Fermat-Weber Points of a Data Matrix}]
For a given data matrix $X_{m\times n}$, {suppose} the $i$-th row of $X$ is ${\bf x}^{(i)}$ ($i\in\{1,\dots,m\}$). We define the {\em Fermat-Weber point of} $X$ as the Fermat-Weber point of $\{{\bf x}^{(1)},\dots,{\bf x}^{(m)}\}$. We {still} denote by $F_X$ the Fermat-Weber point of $X$.
\end{definition}

\begin{definition}[{\bf Projection Matrix}]
\label{def:projectionMatrix}
For a given data matrix $X_{m\times n}$, and for any two fixed integers $d_1\;\text{and}\;d_2\;(2\leq d_1<d_2\leq n)$,  we define the {\em projection matrix} of $X$ ({denoted by $P_{d_1,d_2}(X)$}) as a matrix with size $m\times n$, such that for all $k\in\{1,\dots,m\}$, the $k\text{-}th$ row of $P_{d_1,d_2}(X)$ is
\begin{equation}
\label{equ:def8}
(0,t,\dots,t,x^{(k)}_{d_1},t,\dots,t,x^{(k)}_{d_2},t,\dots,t),
\end{equation}
where
\begin{itemize}
    \item $x^{(k)}_{d_1}$ and $x^{(k)}_{d_2}$ are respectively the $(k,d_1)$-entry and the $(k,d_2)$-entry of $X$, and are respectively located at the $(k,d_1)$-entry and the $(k,d_2)$-entry of of $P_{d_1,d_2}(X)$.
    \item $t$ is a fixed number, such that $t=\min\limits_{1\leq k\leq m}\min\limits_{1\leq \ell\leq n}\{(k,\ell)\text{-entry of}\;X\}$.
\end{itemize}
\end{definition}

Note that the projection matrix $P_{d_1,d_2}(X)$ is still a data matrix.

{Recall the Proposition \ref{prop:LinBo} tells that, if $F_X$ is a Fermat-Weber point of $X=\{{\bf x}^{(1)},\dots,{\bf x}^{(m)}\}\subset\mathbb{R}^{n}\!/\mathbb{R}\mathbf{1}$, then $F_X$ is the unique Fermat-Weber point of $\{F_X,{\bf x}^{(1)},\dots,{\bf x}^{(m)}\}$.}

\begin{theorem}
\label{theo:algorithmIsRight}
Suppose we have a data matrix $X_{(m+1)\times n}$, where the last row of $X$ is a Fermat-Weber point of the matrix made by the first $m$ rows of $X$. We fix two integers $d_1$ and $d_2\;(2\leq d_1<d_2\leq n)$. Let ${\bf r}$ be the last row of $P_{d_1,d_2}(X)$.

If ${\bf r}$ is a Fermat-Weber point of $P_{d_1,d_2}(X)$, and ${\bf u}^{(1)}$, ${\bf u}^{(2)}$ and ${\bf u}^{(3)}$ are defined by \eqref{Equ:0}-\eqref{Equ:-1}, then the projection of {the} Fermat-Weber point of $X$ on $tconv(\{{\bf u}^{(1)},{\bf u}^{(2)},{\bf u}^{(3)}\})$ is a Fermat-Weber point of the projection of $X$ on $tconv(\{{\bf u}^{(1)},{\bf u}^{(2)},{\bf u}^{(3)}\})$.
\end{theorem}

\begin{proof}
By Lemma \ref{lem:projToTriangleBelow(manyPoints)} and Definition \ref{def:projectionMatrix} we know that, the projection of $X$ on $${\mathcal C}:=tconv(\{{\bf u}^{(1)},{\bf u}^{(2)},{\bf u}^{(3)}\})$$
is $P_{d_1,d_2}(X)$. {Note that the last row of $X$ is the unique Fermat-Weber point of $X$. Also note that ${\bf r}$ is the projection of the last row of $X$.} Then by the assumption that ${\bf r}$ is a Fermat-Weber point of $P_{d_1,d_2}(X)$ we know that, the projection of {the} Fermat-Weber point of $X$ on ${\mathcal C}$ is a Fermat-Weber point of the projection of $X$ on ${\mathcal C}$.
\end{proof}

\begin{theorem}
\label{lem:cannotBothHold}
Suppose we have a data matrix $X_{m\times n}$. We fix two integers $d_1$ and $d_2\;(2\leq d_1<d_2\leq n)$. Let ${\bf r}$ be a point
$$(0,t,\dots,t,r_{d_1},t,\dots,t,r_{d_2},t,\dots,t),$$
{where $r_{d_1}$ and $r_{d_2}$ are undetermined numbers, and $t$ is the smallest entry of $X$.} Let ${\bf f}$ be a Fermat-Weber point of $P_{d_1,d_2}(X)$.
If
$r_{d_1}=f_{d_1},$ and $r_{d_2}=f_{d_2},$
then ${\bf r}$ is a Fermat-Weber point of $P_{d_1,d_2}(X)$.
\end{theorem}

\begin{proof}
By Definition \ref{def:projectionMatrix}, the $i$-th row of $P_{d_1,d_2}(X)$ has the form
$${\bf p^{(i)}}:=(0,t,\dots,t,x^{(i)}_{d_1},t,\dots,t,x^{(i)}_{d_2},t,\dots,t),\;\text{for all}\; i\in\{1,\dots,m\}.$$
Assume that
${\bf f}=(0,f_2,\dots,f_{n})$
is a Fermat-Weber point of $P_{d_1,d_2}(X)$. Suppose there exists $k\in S:=\{1,\dots,n\}\backslash \{1,d_1,d_2\},$ such that $f_{k}\neq t.$
For any $i\in\{1,\dots,m\}$, let
\begin{align*}
&A_i=\min\{0,f_{d_1}-x^{(i)}_{d_1},f_{d_2}-x^{(i)}_{d_2}\},\\
&B_i=\max\{0,f_{d_1}-x^{(i)}_{d_1},f_{d_2}-x^{(i)}_{d_2}\}.
\end{align*}
Then we have
\begin{align*}
\sum\limits_{i=1}^{m}d_{tr}({\bf p}^{(i)},{\bf r})=\sum\limits_{i=1}^{m}{(B_i-A_i)},\;\;
\sum\limits_{i=1}^{m}d_{tr}({\bf p}^{(i)},{\bf f})=\sum\limits_{i=1}^{m}({\max\limits_{k\in S}\{B_i,f_k-t\}-\min\limits_{k\in S}\{A_i,f_k-t\}}).
\end{align*}
It is easy to see that $\sum\limits_{i=1}^{m}d_{tr}({\bf p}^{(i)},{\bf f})\geq\sum\limits_{i=1}^{m}d_{tr}({\bf p}^{(i)},{\bf r}).$ So, by Definition \ref{def:FWP}, ${\bf r}$ is a Fermat-Weber point of $P_{d_1,d_2}(X).$
\end{proof}

\section{algorithms}\label{sec:algorithm}
In this section, we develop Algorithm \ref{algo:MofE} and Algorithm \ref{algo:NewMethod}, such that for a given data set $X\subset\mathbb{R}^{n}\!/\mathbb{R}\mathbf{1}$, these two algorithms output a tropical triangle ${\mathcal C}$, on which the projection of a Fermat-Weber point of $X$ is a Fermat-Weber point of the projection of $X$.

The input of Algorithm \ref{algo:MofE} and Algorithm \ref{algo:NewMethod} is a data set $$\{{\bf x}^{(1)},\dots,{\bf x}^{(m)}\}\subset\mathbb{R}^{n}\!/\mathbb{R}\mathbf{1}.$$Algorithm \ref{algo:MofE} and Algorithm \ref{algo:NewMethod} output three points
$${\bf u}^{(1)},{\bf u}^{(2)},{\bf u}^{(3)}\in\mathbb{R}^{n}\!/\mathbb{R}\mathbf{1},$$
such that the projection of a Fermat-Weber point of $\{{\bf x}^{(1)},\dots,{\bf x}^{(m)}\}$ on
$${\mathcal C}:=tconv(\{{\bf u}^{(1)},{\bf u}^{(2)},{\bf u}^{(3)}\})$$
is a Fermat-Weber point of the projection of $\{{\bf x}^{(1)},\dots,{\bf x}^{(m)}\}$ on ${\mathcal C}$.

There are two main steps in each algorithm as follows.

\noindent
{\bf Step 1.}\;We define a data matrix $X$, such that for all $i\in\{1,\dots,m\},$ the $i$-th row of $X$ is ${\bf x}^{(i)}.$ We obtain a Fermat-Weber point $F_X$ by solving the linear programming \eqref{equ:FWpointLP}. We define a matrix ${\widetilde X}$ with size $(m+1)\times n$, such that the last row of ${\widetilde X}$ is $F_X$, and the first $m$ rows of ${\widetilde X}$ {come from} $X$.

\noindent
{\bf Step 2.}\;We traverse all pairs $(d_1,d_2)$ such that $2\leq d_1<d_2\leq n$, {and we} calculate the projection matrix $P_{d_1,d_2}({\widetilde X})$ by Definition \ref{def:projectionMatrix}. Check if the last row of $P_{d_1,d_2}({\widetilde X})$ is a Fermat-Weber point of $P_{d_1,d_2}({\widetilde X})$. If so, {we} calculate the three points ${\bf u}^{(1)},{\bf u}^{(2)}\;\text{and}\;{\bf u}^{(3)}$ by \eqref{Equ:0}-\eqref{Equ:-1} in Lemma \ref{lem:projToTriangleBelow(manyPoints)}, return the output, and terminate. By Theorem \ref{theo:algorithmIsRight} we know that, the projection of a Fermat-Weber point of ${X}$ on ${\mathcal C}=tconv(\{{\bf u}^{(1)},{\bf u}^{(2)},{\bf u}^{(3)}\})$ is a Fermat-Weber point of the projection of ${X}$ on ${\mathcal C}.$ If for all $(d_1,d_2)$, the last row of $P_{d_1,d_2}({\widetilde X})$ is not a Fermat-Weber point of $P_{d_1,d_2}({\widetilde X})$, then return FAIL.

\begin{remark}
\label{rem:SorFsimultaneously}
It is not guaranteed that Algorithm \ref{algo:MofE} and Algorithm \ref{algo:NewMethod} will always succeed ({return the tropical triangle}). {If the algorithms succeed, then by Theorem \ref{theo:algorithmIsRight}, Algorithm \ref{algo:MofE} is correct, and by Theorem \ref{theo:algorithmIsRight} and Theorem \ref{lem:cannotBothHold}, Algorithm \ref{algo:NewMethod} is correct.}

Algorithm \ref{algo:MofE} and Algorithm \ref{algo:NewMethod} always succeed or fail simultaneously. But our experimental results in the next section show that, {Algorithm \ref{algo:MofE} or Algorithm \ref{algo:NewMethod} succeeds with a much higher probability than choosing tropical triangles randomly (see Table \ref{tab:showMyAlgoIsGood} and Table \ref{tab:goodRateForMNchange}). Our experimental results also show that,} if Algorithm \ref{algo:MofE} and Algorithm \ref{algo:NewMethod} succeed, then with the {probability} more than $50\%$, Algorithm \ref{algo:NewMethod} would terminate in less traversal steps than Algorithm \ref{algo:MofE} does (see Figure \ref{fig:VarChangingCompare}).
\end{remark}

Remark that, {the difference between Algorithm \ref{algo:MofE} and Algorithm \ref{algo:NewMethod} is the traversal strategy, i.e., the Step 2. is different.} Below we give more details about Step 2.

Let
\begin{equation}\label{equ:defTheL}
L=\{(2,3), (2,4),\dots, (2,n), (3,4), (3,5),\dots, (3,n),\dots,(n-1,n)\}.
\end{equation}
\begin{itemize}
\item[1.]{In Algorithm \ref{algo:MofE}: Step 2.,we traverse all pairs $(d_1,d_2)$ ($2\leq d_1<d_2\leq n$) {in $L$ one by one, i.e., we traverse the pairs in the lexicographical order.}}

\item[2.]{In Algorithm \ref{algo:NewMethod}: Step 2., {we consider the same $L$ defined in \eqref{equ:defTheL}.} Note that $|L|=\frac{(n-1)(n-2)}{2}$. Let $W$ and $S$ be two empty sets. In the future, we will record in $W$ some indices that will be traversed {in priority}, and record in $S$ the pairs that have been traversed. Let ${\bf u}^{(1)}$, ${\bf u}^{(2)}$ and ${\bf u}^{(3)}$ be null vectors.

     Now we start a loop ({see lines \ref{A2code:beginLoop}-\ref{A2code:endLoop} in Algorithm \ref{algo:NewMethod}}). In this loop, we traverse all pairs in $L$ while $|S|<\frac{(n-1)(n-2)}{2}$, and ${\bf u}^{(1)}$, ${\bf u}^{(2)}$ and ${\bf u}^{(3)}$ are null. For each pair $(d_1,d_2)\in L$, if $(d_1,d_2)\in S$, then we skip the pair. If $(d_1,d_2)\notin S$, then we add the pair into $S$, and calculate the projection matrix $P_{d_1,d_2}({\widetilde X})$ by Definition \ref{def:projectionMatrix}. Let ${\bf r}$ be the last row of $P_{d_1,d_2}({\widetilde X})$. If ${\bf r}$ is a Fermat-Weber point of $P_{d_1,d_2}({\widetilde X})$, then we calculate ${\bf u}^{(1)}$, ${\bf u}^{(2)}$ and ${\bf u}^{(3)}$ by formulas \eqref{Equ:0}-\eqref{Equ:-1} in Lemma \ref{lem:projToTriangleBelow(manyPoints)}, return the output and terminate. By Theorem \ref{theo:algorithmIsRight} we know that, the projection of a Fermat-Weber point of ${X}$ on ${\mathcal C}:=tconv(\{{\bf u}^{(1)},{\bf u}^{(2)},{\bf u}^{(3)}\})$ is a Fermat-Weber point of the projection of ${X}$ on ${\mathcal C}.$ If ${\bf r}$ is not a Fermat-Weber point of $P_{d_1,d_2}({\widetilde X})$, then by Theorem \ref{lem:cannotBothHold}, at most one of the following two equalities holds:
\begin{align}
&r_{d_1}=f_{d_1},\label{equ:judgePause1}\\
&r_{d_2}=f_{d_2}\label{equ:judgePause2},
\end{align}
{where ${\bf f}$ is a Fermat-Weber point of $P_{d_1,d_2}({\widetilde X}).$}
So we have $3$ cases.
 \begin{itemize}
    \item[]{{\bf (Case 1)} If only \eqref{equ:judgePause1} holds, then we add $d_1$ into $W$, and stop doing the traversal of $L$.}
    \item[]{{\bf (Case 2)} If only \eqref{equ:judgePause2} holds, then we add $d_2$ into $W$, and stop doing the traversal of $L$.}
    \item[]{{\bf (Case 3)} If neither \eqref{equ:judgePause1} nor \eqref{equ:judgePause2} holds, then we move on to the next pair in $L$.}
 \end{itemize}
  Now we explain what we do if {\bf(Case 1)} happens ({\bf(Case 2)} is similar). Note that $W$ is nonempty {at this time}, and ${\bf u}^{(1)}$, ${\bf u}^{(2)}$ and ${\bf u}^{(3)}$ are null. For each element $\omega\in W$, we define
 \begin{equation}
 \label{equ:L_omega}
 L_{\omega}=\{(\omega_1,\omega_2)\in L|\omega_1=\omega\;\text{or}\;\omega_2=\omega\}.
 \end{equation}
 We start traversing all pairs in $L_{\omega}$. For each pair $(\omega_1,\omega_2)\in L_{\omega}$, if $(\omega_1,\omega_2)\in S$, then we skip the pair. If $(\omega_1,\omega_2)\notin S$, then we add the pair into $S$, and calculate the projection matrix $P_{\omega_1,\omega_2}({\widetilde X})$ by Definition \ref{def:projectionMatrix}. Let ${\bf r}$ be the last row of $P_{\omega_1,\omega_2}({\widetilde X})$. If ${\bf r}$ is a Fermat-Weber point of $P_{\omega_1,\omega_2}({\widetilde X})$, then calculate ${\bf u}^{(1)}$, ${\bf u}^{(2)}$ and ${\bf u}^{(3)}$ by formulas \eqref{Equ:0}-\eqref{Equ:-1} in Lemma \ref{lem:projToTriangleBelow(manyPoints)}, output ${\bf u}^{(1)}$, ${\bf u}^{(2)}$ and ${\bf u}^{(3)}$, and terminate. If ${\bf r}$ is not a Fermat-Weber point of $P_{\omega_1,\omega_2}({\widetilde X})$, then by Theorem \ref{lem:cannotBothHold}, at most one of the following two equalities holds:
  \begin{align}
    &r_{\omega_1}=f_{\omega_1},\label{equ:judgePause3}\\
    &r_{\omega_2}=f_{\omega_2}\label{equ:judgePause4},
\end{align}
{where ${\bf f}$ is a Fermat-Weber point of $P_{\omega_1,\omega_2}({\widetilde X}).$} So we have $2$ cases.
\begin{itemize}
    \item[]{{\bf (Case 1.1)} If only \eqref{equ:judgePause3} holds, then add $\omega_1$ into $W$.}
    \item[]{{\bf (Case 1.2)} If only \eqref{equ:judgePause4} holds, then add $\omega_2$ into $W$.}
\end{itemize}
  We move on to the next pair in $L_{\omega}$. If for any pair $(\omega_1,\omega_2)\in L$, we have $(\omega_1,\omega_2)\in S$, and the last row of $P_{\omega_1,\omega_2}({\widetilde X})$ is not a Fermat-Weber point of $P_{\omega_1,\omega_2}({\widetilde X})$, then we remove this $\omega$ from $W$. If $W$ becomes empty again, then we continue the traversal of $L$ we paused in {\bf (Case 1)} (Page 10). If $W$ is still nonempty after one element in $W$ has been removed, then for the next $\omega\in W,$ we traverse $L_{\omega}$.}
\end{itemize}

Now we give two examples to better explain how Algorithm \ref{algo:MofE} and Algorithm \ref{algo:NewMethod} work.

\begin{example}
\label{ex:Algorithm1}
This example explains how Algorithm \ref{algo:MofE} works. Suppose we have a {data matrix}
\begin{equation*}
X=\left(
\begin{array}{rrrrr}
0&211&45&-33&10\\
0&-365&23&35&64\\
0&-40&-59&63&14\\
0&65&257&39&-35\\
0&13&5&-261&21\\
0&-1&91&355&7\\
0&-644&21&58&36\\
0&59&4&362&15
\end{array}
\right).
\end{equation*}
By running the package {\tt lpSolve} \cite{lpSolve} in {\tt R} to solve the linear programming \eqref{equ:FWpointLP}, we obtain a Fermat-Weber point of $X$, which is
$$F_X=(0,-40,4,89,15).$$
Define a matrix ${\widetilde X}$ with size $(m+1)\times n$, such that the last row of ${\widetilde X}$ is $F_X$, and the first $m$ rows of ${\widetilde X}$ {come from} $X$. We have
\begin{equation*}
{\widetilde X}=\left(
\begin{array}{rrrrr}
0&211&45&-33&10\\
0&-365&23&35&64\\
0&-40&-59&63&14\\
0&65&257&39&-35\\
0&13&5&-261&21\\
0&-1&91&355&7\\
0&-644&21&58&36\\
0&59&4&362&15\\
0&-40&4&89&15
\end{array}
\right).
\end{equation*}
Now we start traversing all pairs $(d_1,d_2)(2\leq d_1<d_2\leq 5)$ in $L$, {where $$L=\{(2,3),(2,4),(2,5),(3,4),(3,5),(4,5)\}.$$}

{\bf 1.} The first pair is $(2,3)$. Note that by Definition \ref{def:projectionMatrix} we have
\begin{equation*}
    P_{2,3}({\widetilde X})=\left(
    \begin{array}{rrrrr}
    0&211&45&-644&-644\\
    0&-365&23&-644&-644\\
    0&-40&-59&-644&-644\\
    0&65&257&-644&-644\\
    0&13&5&-644&-644\\
    0&-1&91&-644&-644\\
    0&-644&21&-644&-644\\
    0&59&4&-644&-644\\
    0&-40&4&-644&-644
    \end{array}
    \right).
\end{equation*}
{And we can compute a Fermat-Weber point of $P_{2,3}({\widetilde X})$: $F_{P_{2,3}({\widetilde X})}=(0,-23,21,-644,-644).$}
The last row of $P_{2,3}({\widetilde X})$ is ${{\bf r}=}(0,-40,4,-644,-644)$. By {Definition \ref{def:FWP}} we can check that, ${\bf r}$ is not a Fermat-Weber point of $P_{2,3}({\widetilde X})$. We move on to the next pair.

{\bf 2.} Similarly, we pass $(2,4)$, $(2,5)$ and $(3,4)$. {For the pair} $(3,5)$, note that
\begin{equation*}
    P_{3,5}({\widetilde X})=\left(
    \begin{array}{rrrrr}
    0&-644&45&-644&10\\
    0&-644&23&-644&64\\
    0&-644&-59&-644&14\\
    0&-644&257&-644&-35\\
    0&-644&5&-644&21\\
    0&-644&91&-644&7\\
    0&-644&21&-644&36\\
    0&-644&4&-644&15\\
    0&-644&4&-644&15
    \end{array}
    \right).
\end{equation*}
And we can compute a Fermat-Weber point of $P_{3,5}({\widetilde X})$: $F_{P_{3,5}({\widetilde X})}=(0,-644,4,-644,15)$, {which is exactly the last row of $P_{3,5}({\widetilde X})$}. By \eqref{Equ:0}-\eqref{Equ:-1} in Lemma \ref{lem:projToTriangleBelow(manyPoints)}, {we make three points}:
\begin{align*}
{\bf u}^{(1)}&=(0,-644,-60,-644,-36),\\
{\bf u}^{(2)}&=(0,-644,-58,-644,65),\\
{\bf u}^{(3)}&=(0,-644,258,-644,-34).
\end{align*}
Then, output ${\bf u}^{(1)},{\bf u}^{(2)},\;\text{and}\;{\bf u}^{(3)}$, and terminate.
\end{example}

\begin{example}
\label{ex:Algorithm2}
This example explains how Algorithm \ref{algo:NewMethod} works. Suppose we have a data matrix
\begin{equation*}
X=\left(
\begin{array}{rrrrr}
0&211&45&-33&10\\
0&-365&23&35&64\\
0&-40&-59&63&14\\
0&65&257&39&-35\\
0&13&5&-261&21\\
0&-1&91&355&7\\
0&-644&21&58&36\\
0&59&4&362&15
\end{array}
\right).
\end{equation*}
By {solving the linear programming \eqref{equ:FWpointLP}}, we obtain a Fermat-Weber point of $X$, which is
$F_X=(0,-40,4,89,15).$
Define a matrix ${\widetilde X}$ with size $(m+1)\times n$, such that the last row of ${\widetilde X}$ is $F_X$, and the first $m$ rows of ${\widetilde X}$ {come from} $X$. We have
\begin{equation*}
{\widetilde X}=\left(
\begin{array}{rrrrr}
0&211&45&-33&10\\
0&-365&23&35&64\\
0&-40&-59&63&14\\
0&65&257&39&-35\\
0&13&5&-261&21\\
0&-1&91&355&7\\
0&-644&21&58&36\\
0&59&4&362&15\\
0&-40&4&89&15
\end{array}
\right).
\end{equation*}
{Let $L$ be a list that contains all pairs $(d_1,d_2)(2\leq d_1<d_2\leq 5)$ in the lexicographical order, that is}
$L=\{(2,3),(2,4),(2,5),(3,4),(3,5),(4,5)\}.$
{Also let $W$ and $S$ be two empty sets. We will record in $W$ some indices that will be traversed in priority, and record in $S$ the pairs that have been traversed.} Now we start the traversal.

{\bf 1.} We first start traversing pairs in $L$. The first pair in $L$ is $(2,3)$. Add $(2,3)$ into $S$. Note that
    \begin{equation*}
    P_{2,3}({\widetilde X})=\left(
    \begin{array}{rrrrr}
    0&211&45&-644&-644\\
    0&-365&23&-644&-644\\
    0&-40&-59&-644&-644\\
    0&65&257&-644&-644\\
    0&13&5&-644&-644\\
    0&-1&91&-644&-644\\
    0&-644&21&-644&-644\\
    0&59&4&-644&-644\\
    0&-40&4&-644&-644
    \end{array}
    \right).
    \end{equation*}
    {And we can compute a Fermat-Weber point of $P_{2,3}({\widetilde X})$: ${\bf f}=(0,-23,21,-644,-644).$}
The last row of $P_{2,3}({\widetilde X})$ is ${\bf r}=(0,-40,4,-644,-644)$. By {Definition \ref{def:FWP}} we can check that, ${\bf r}$ is not a Fermat-Weber point of $P_{2,3}({\widetilde X})$. We have
$r_2=-40\neq -23=f_2,$ and
$r_3=4\neq 21=f_3.$
Now {{\bf (Case 3)}} in page 10 happens, so we move on to the next pair in $L$.

{\bf 2.} The next pair in $L$ is $(2,4)$. Add $(2,4)$ into $S$. Note that
    \begin{equation*}
    P_{2,4}({\widetilde X})=\left(
    \begin{array}{rrrrr}
    0&211&-644&-33&-644\\
    0&-365&-644&35&-644\\
    0&-40&-644&63&-644\\
    0&65&-644&39&-644\\
    0&13&-644&-261&-644\\
    0&-1&-644&355&-644\\
    0&-644&-644&58&-644\\
    0&59&-644&362&-644\\
    0&-40&-644&89&-644
    \end{array}
    \right).
    \end{equation*}
    {And we can compute a Fermat-Weber point of $P_{2,4}({\widetilde X})$: ${\bf f}=(0,-40,-644,63,-644)$.}  The last row of $P_{2,4}({\widetilde X})$ is ${\bf r}=(0,-40,-644,89,-644)$. {By Definition \ref{def:FWP}} we can check that, ${\bf r}$ is not a Fermat-Weber point of $P_{2,4}({\widetilde X})$. We have
$r_2=-40=f_2,$ and $r_4=63\neq 89=f_4.$
Now {{\bf (Case 1)}} in Page 10 happens, so we add $2$ into $W$, and pause the traversal in $L$. Note that, now $W{=\{2\}}$ is nonempty, and the first element in $W$ is $2$. By \eqref{equ:L_omega}, we have
    $L_2=\{(2,3),(2,4),(2,5)\}.$
    We start traversing pairs in $L_2.$

{\bf3.} {Note that now $S=\{(2,3),(2,4)\}.$}  The {first} pair in $L_2$ is $(2,3)$, which is in $S$ already, so we skip it. Similarly we skip $(2,4)$. The {third} pair in $L_2$ is $(2,5)$, which is not in $S$, so we do the following steps. Add $(2,5)$ into $S$. Note that
    \begin{equation*}
    P_{2,5}({\widetilde X})=\left(
    \begin{array}{rrrrr}
    0&211&-644&-644&10\\
    0&-365&-644&-644&64\\
    0&-40&-644&-644&14\\
    0&65&-644&-644&-35\\
    0&13&-644&-644&21\\
    0&-1&-644&-644&7\\
    0&-644&-644&-644&36\\
    0&59&-644&-644&15\\
    0&-40&-644&-644&15
    \end{array}
    \right).
    \end{equation*}
    {And we can compute a Fermat-Weber point of $P_{2,5}({\widetilde X})$: ${\bf f}=(0,-1,-644,-644,15)$.}
    The last row of $P_{2,5}({\widetilde X})$ is ${\bf r}=(0,-40,-644,-644,15)$. By {Definition \ref{def:FWP}} we can check that, ${\bf r}$ is not a Fermat-Weber point of $P_{2,5}({\widetilde X})$. We have
    $r_2=-40\neq-1=f_2,$ and $r_5=15=f_5.$
Now {{\bf (Case 1.2)}} in Page 11 happens, so we add $5$ into $W$, {and now $W=\{2,5\}$}. {Note that $S=\{(2,3),(2,4),(2,5)\}$.} Since for every pair $(\omega_1,\omega_2)\in L_2$, {$(\omega_1,\omega_2)$ is in $S$}, and the last row of $P_{\omega_1,\omega_2}({\widetilde X})$ is not a Fermat-Weber point of $P_{\omega_1,\omega_2}({\widetilde X})$, we remove $2$ from $W$.

{\bf 4.} Note that, now $W=\{5\}$ is nonempty. By \eqref{equ:L_omega}, we have
    $L_5=\{(2,5),(3,5),(4,5)\}.$
    The {first} pair in $L_5$ is $(2,5)$, which is in $S$ already, so we skip it. The {second} pair in $L_5$ is $(3,5)$, which is not in $S$, so we do the following steps. Add $(3,5)$ into $S$. Note that
    \begin{equation*}
    P_{3,5}({\widetilde X})=\left(
    \begin{array}{rrrrr}
    0&-644&45&-644&10\\
    0&-644&23&-644&64\\
    0&-644&-59&-644&14\\
    0&-644&257&-644&-35\\
    0&-644&5&-644&21\\
    0&-644&91&-644&7\\
    0&-644&21&-644&36\\
    0&-644&4&-644&15\\
    0&-644&4&-644&15
    \end{array}
    \right).
    \end{equation*}
   {And we can compute a Fermat-Weber point of $P_{3,5}({\widetilde X})$: ${\bf f}=(0,-644,4,-644,15)$,} which is the last row of $P_{3,5}({\widetilde X})$. By \eqref{Equ:0}-\eqref{Equ:-1} in Lemma \ref{lem:projToTriangleBelow(manyPoints)}, we make three points:
    ${\bf u}^{(1)}=(0,-644,-60,-644,-36),$
    ${\bf u}^{(2)}=(0,-644,-58,-644,65),$ and
    ${\bf u}^{(3)}=(0,-644,258,-644,-34).$
   Then, output ${\bf u}^{(1)},{\bf u}^{(2)},\;\text{and}\;{\bf u}^{(3)}$, and terminate.
\end{example}

Below we give the pseudo code of Algorithm \ref{algo:MofE} and Algorithm \ref{algo:NewMethod}.
Note that, Algorithm \ref{algo:obtainTropicalDistanceSum} and Algorithm \ref{algo:obtainFeasibleTriangle} are sub-algorithms of Algorithm \ref{algo:MofE} and Algorithm \ref{algo:NewMethod}. {For a given data matrix $X$,} Algorithm \ref{algo:obtainTropicalDistanceSum} calculates the summation of tropical distance between the last row of $X$ and each row of $X$, and also calculates the summation of tropical distance between a Fermat-Weber point of $X$ and each row of $X$. {We will use Algorithm \ref{algo:obtainTropicalDistanceSum} to check if the last row of $X$ is a Fermat-Weber point of $X$.} Algorithm \ref{algo:obtainFeasibleTriangle} calculates three points ${\bf u}^{(1)}$, ${\bf u}^{(2)}$ and ${\bf u}^{(3)}$ by \eqref{Equ:0}-\eqref{Equ:-1}.
\newpage
\begin{algorithm}[h]
    \caption{}
    \label{algo:MofE}
    \LinesNumbered
    \KwIn{$\{{\bf x}^{(1)},\dots,{\bf x}^{(m)}\}\subset\mathbb{R}^n\!/\mathbb{R}\mathbf{1}$}
    \KwOut{${\bf u}^{(1)},{\bf u}^{(2)},{\bf u}^{(3)}$,
    \par where ${\bf u}^{(1)}$, ${\bf u}^{(2)}$ and ${\bf u}^{(3)}$ are three points in $\mathbb{R}^{n}\!/\mathbb{R}\mathbf{1}$, such that the projection of a Fermat-Weber point of $\{{\bf x}^{(1)},\dots,{\bf x}^{(m)}\}$ on ${\mathcal C}:=tconv(\{{\bf u}^{(1)},{\bf u}^{(2)},{\bf u}^{(3)}\})$ is a Fermat-Weber point of the projection of $\{{\bf x}^{(1)},\dots,{\bf x}^{(m)}\}$ on ${\mathcal C}$}
    \par $X_{m\times n}\leftarrow\{{\bf x}^{(1)},\dots,{\bf x}^{(m)}\}$,\quad $F_X\leftarrow$ a Fermat-Weber point of $X_{m\times n}$\label{A1code:getFWP}
    \par $X_{(m+1)\times n}\leftarrow$ the last row is $F_X$, and the first $m$ rows {come from} $X_{m\times n}$
    \par${\bf u}^{(1)},{\bf u}^{(2)},{\bf u}^{(3)}\leftarrow n$-dimensional null vectors
    \par\For{$d_1\text{ from }2\text{ to }n-1$}{
        \For{$d_2\text{ from }d_1+1\text{ to }n$}{
            \par\If{Verify-FW-Point($P_{d_1,d_2}(X)$)=TRUE}{
                \par${\bf u}^{(1)},{\bf u}^{(2)},{\bf u}^{(3)}\leftarrow$ {Compute-Triangle}($X,d_1,d_2$)
            }

            \par{\bf if} ${\bf u}^{(1)},{\bf u}^{(2)}\;\text{and}\;{\bf u}^{(3)}$ are not null {\bf then} break
        }
    }
    \par{\bf if} ${\bf u}^{(1)},{\bf u}^{(2)}\;\text{and}\;{\bf u}^{(3)}$ are null {\bf then return} FAIL, {\bf otherwise}, {\bf return} ${\bf u}^{(1)},{\bf u}^{(2)},{\bf u}^{(3)}$
\end{algorithm}

\begin{algorithm}[H]
\caption{{Verify-FW-Point}}
\label{algo:obtainTropicalDistanceSum}
\LinesNumbered
\KwIn{Data matrix $X_{m\times n}$}
\KwOut{{TRUE, if the last row of $X$ is a Fermat-Weber point of $X$; FALSE, if the last row of $X$ is not a Fermat-Weber point of $X$}}
    \par${\bf r}\leftarrow$ the last row of $X$, \quad${\bf f}\leftarrow$ a Fermat-Weber point of $X$
    \par$d_{\bf r}\leftarrow\sum\limits_{i=1}^{m}d_{tr}({\bf r},{\bf x}^{(i)})$, \quad$d_{\bf f}\leftarrow\sum\limits_{i=1}^{m}d_{tr}({\bf f},{\bf x}^{(i)})$,\quad {where ${\bf x}^{(i)}$ is the $i$-th row of $X$}
    \par {{\bf if} $d_{\bf r}=d_{\bf f}$ {\bf then return} TRUE, {\bf otherwise}, {\bf return} FALSE}
\end{algorithm}

\begin{algorithm}[H]
\caption{{Compute-Triangle}}
\label{algo:obtainFeasibleTriangle}
\LinesNumbered
\KwIn{$\text{Data matrix}\;X_{m\times n},\text{ and two indices }d_1,d_2$}
\KwOut{${\bf u}^{(1)},{\bf u}^{(2)},{\bf u}^{(3)}$,
    \par where ${\bf u}^{(1)}$, ${\bf u}^{(2)}$ and ${\bf u}^{(3)}$ are {defined by \eqref{Equ:0}-\eqref{Equ:-1}}}
    \par${\bf u}^{(1)},{\bf u}^{(2)},{\bf u}^{(3)}\leftarrow  n$-dimensional null vectors, \quad$X_{min}\leftarrow$ the smallest entry of $X$
    \par {$v^{(S)}_1,v^{(S)}_2\leftarrow$ the smallest coordinates in the $d_1$-th and $d_2$-th columns of $X$ respectively}
    \par {$v^{(L)}_1,v^{(L)}_2\leftarrow$ the largest coordinates in the $d_1$-th and $d_2$-th columns of $X$ respectively}
    \par$u^{(i)}_1\leftarrow0$ for $i=1,2,3$
    \par$u^{(1)}_{d_1}\leftarrow v^{(S)}_{1}-1$, \quad$u^{(1)}_{d_2}\leftarrow v^{(S)}_{2}-1$, \quad$u^{(2)}_{d_1}\leftarrow v^{(S)}_{1}+1$, \quad$u^{(2)}_{d_2}\leftarrow v^{(L)}_{2}+1$
    \par$u^{(3)}_{d_1}\leftarrow v^{(L)}_{1}+1$, \quad$u^{(3)}_{d_2}\leftarrow v^{(S)}_{2}+1$, \quad$\text{all other coordinates of }{\bf u}^{(1)},{\bf u}^{(2)},{\bf u}^{(3)}\leftarrow X_{min}$

\par \Return ${\bf u}^{(1)},{\bf u}^{(2)},{\bf u}^{(3)}$
\end{algorithm}

\newpage
\begin{algorithm}[H]
    \caption{}
    \label{algo:NewMethod}
    \LinesNumbered
    \KwIn{$\{{\bf x}^{(1)},\dots,{\bf x}^{(m)}\}\subset\mathbb{R}^n\!/\mathbb{R}\mathbf{1}$}
    \KwOut{${\bf u}^{(1)},{\bf u}^{(2)},{\bf u}^{(3)}$,
    \par where ${\bf u}^{(1)}$, ${\bf u}^{(2)}$ and ${\bf u}^{(3)}$ are three points in $\mathbb{R}^{n}\!/\mathbb{R}\mathbf{1}$, such that the projection of a Fermat-Weber point of $\{{\bf x}^{(1)},\dots,{\bf x}^{(m)}\}$ on ${\mathcal C}:=tconv(\{{\bf u}^{(1)},{\bf u}^{(2)},{\bf u}^{(3)}\})$ is a Fermat-Weber point of the projection of $\{{\bf x}^{(1)},\dots,{\bf x}^{(m)}\}$ on ${\mathcal C}$}
    \par $X_{m\times n}\leftarrow\{{\bf x}^{(1)},\dots,{\bf x}^{(m)}\}$, \quad$F_X\leftarrow$ a Fermat-Weber point of $X_{m\times n}$\label{A2code:getFWP}
    \par$X_{(m+1)\times n}\leftarrow$ the last row is $F_X$, and the first $m$ rows {come from} $X_{m\times n}$
    \par $L\leftarrow\text{all pairs of indices } {(d_1,d_2) (2\leq d_1<d_2\leq n)}\text{ in the lexicographical order,} $
    \newline that is: $\{(2,3),(2,4),\dots,(2,n),(3,4),\dots,(3,n),\dots,(n-1,n)\}$
    \par$S\leftarrow\varnothing$ {(we will record in $S$ the pairs that have been traversed)}
    \par$W\leftarrow\varnothing\text{ (we will record in }W\text{ the set of indices that will be traversed in priority)}$
    \par$t\leftarrow0$,\quad ${\bf u}^{(1)},{\bf u}^{(2)},{\bf u}^{(3)}\leftarrow n$-dimensional null vectors
    \par\While{$|S|<\frac{(n-1)(n-2)}{2}$, and ${\bf u}^{(1)},{\bf u}^{(2)}\;\text{and}\;{\bf u}^{(3)}$ are null\label{A2code:beginLoop}}{
        \par $t\leftarrow t+1$
        \par{\bf if} {$L[t]\in S$ (here, $L[t]$ denotes the $i$-th element of $L$)} {\bf then} skip this round
        \par$S\leftarrow S\cup\{L[t]\}$, \quad$(d_1,d_2)\leftarrow L[t]$
        \par${\bf r}\leftarrow$ the $(m+1)$-th row of $P_{d_1,d_2}(X)$, \quad${\bf f}\leftarrow F_{P_{d_1,d_2}(X)}$
            \par\If{{Verify-FW-Point($P_{d_1,d_2}(X)$)=TRUE}}{
                \par${\bf u}^{(1)},{\bf u}^{(2)},{\bf u}^{(3)}\leftarrow$ {Compute-Triangle}($X,d_1,d_2$)
            }
        \par{\bf if} ${\bf u}^{(1)},{\bf u}^{(2)}\;\text{and}\;{\bf u}^{(3)}$ are not null {\bf then} break
        \par\If{$f_{k}=r_{k}$ for $k=d_1\text{ or }d_2$}{
            \par$W\leftarrow W\cup \{k\}$
            \par\While{$W\neq\varnothing$ and ${\bf u}^{(1)},{\bf u}^{(2)}\;\text{and}\;{\bf u}^{(3)}$ are null}{
                \par$\omega\leftarrow W[1]$ ({here, $W[1]$ denotes} the first element in $W$)
                \par$L_{\omega}\leftarrow\{(\omega_1,\omega_2)\in L|\omega=\omega_1\;\text{or}\;\omega=\omega_2\}$
                \par\For{all $(\omega_1,\omega_2)\in L_{\omega}$ such that $(\omega_1,\omega_2)\notin S$}{
                    \par$S\leftarrow S\cup\{(\omega_1,\omega_2)\}$, \quad${\bf r}\leftarrow$ the $(m+1)$-th row of $P_{\omega_1,\omega_2}(X)$
                    \par${\bf f}\leftarrow F_{P_{\omega_1,\omega_2}(X)}$
                    \par\If{{Verify-FW-Point($P_{\omega_1,\omega_2}(X)$)=TRUE}}{
                \par${\bf u}^{(1)},{\bf u}^{(2)},{\bf u}^{(3)}\leftarrow$ {Compute-Triangle}($X,\omega_1,\omega_2$)
            }
                    \par{\bf if} ${\bf u}^{(1)},{\bf u}^{(2)}\;\text{and}\;{\bf u}^{(3)}$ is not null {\bf then} break
                    \par{\bf if} $f_{j}=r_{j}$ for $j=\omega_1\text{ or }\omega_2$, and $j\notin W$ {\bf then} $W\leftarrow W\cup \{j\}$

                }
                \par$W\leftarrow W\backslash W[1]$\label{A2code:endLoop}
            }
        }
    }
    \par{\bf if} ${\bf u}^{(1)},{\bf u}^{(2)}\;\text{and}\;{\bf u}^{(3)}$ are null {\bf then return} FAIL, {\bf otherwise}, {\bf return} ${\bf u}^{(1)},{\bf u}^{(2)},{\bf u}^{(3)}$
\end{algorithm}

\newpage
\section{Implementation and Experiment}\label{sec:experiment}
We implement Algorithm \ref{algo:MofE} and Algorithm \ref{algo:NewMethod}, and test how Algorithm \ref{algo:MofE} and Algorithm \ref{algo:NewMethod} {perform}. Data matrices, {\tt R} code and computational results are available online via: \url{https://github.com/DDDVE/the-Projection-of-Fermat-Weber-Points.git}.

\begin{itemize}
\item[]{\bf Software} We implement Algorithm \ref{algo:MofE} and Algorithm \ref{algo:NewMethod} in {\tt R} (version 4.0.4) \cite{R}, where we use the command {\tt lp()} {in the package {\tt lpSolve} \cite{lpSolve}} to implement Line \ref{A1code:getFWP} in Algorithm \ref{algo:MofE} and Line \ref{A2code:getFWP} in Algorithm \ref{algo:NewMethod} for computing a Fermat-Weber point of a data matrix.

In our experiments, we use the command {\tt rmvnorm()} {in the package {\tt Rfast} \cite{Rfast}} to generate data matrices that obey multivariate normal distribution.

\item[]{\bf Hardware and System} We use a 3.6 GHz Intel Core i9-9900K processor (64 GB of RAM) under Windows 10.
\end{itemize}

Now we present four tables and one figure to illustrate how Algorithm \ref{algo:MofE} and Algorithm \ref{algo:NewMethod} perform.

\begin{itemize}
\item[]{\bf1.} {For a fixed data matrix $X_{m\times n}$,} Table \ref{tab:showMyAlgoIsGood} shows the proportion of random tropical triangles, on which the projection of a Fermat-Weber point of $X$ is a Fermat-Weber point of the projection of $X$. From Table \ref{tab:showMyAlgoIsGood} we can see that, for a fixed data matrix $X$, and for random tropical triangles, the ``succeed rate" is low. Here, by ``succeed rate", we mean the proportion of random tropical triangles, on which the projection of a Fermat-Weber point of $X$ is a Fermat-Weber point of the projection of $X$. For instance, the highest proportion is \textcolor{red}{16\%}, and the lowest proportion is even only \textcolor{red}{1\%}. Besides, the succeed rate is extremely low when $m$ and $n$ are both big.

\begin{table}[H]\scriptsize
    \centering
    \begin{tabular}{|c|c|c|c|c|}
    \hline \diagbox{n}{succeed rate}{m} &   30      &   60      &   90   &  120         \\
    \hline 5                            &   \textcolor{red}{16\%}      &   8\%      &   10\%   &    6\%        \\
    \hline 10                           &   4\%      &   9\%      &   8\%   &    5\%        \\
    \hline 15                           &   11\%      &   5\%      &   2\%   &    \textcolor{red}{1\%}        \\
    \hline 20                           &   8\%      &   6\%      &   \textcolor{red}{1\%}   &    \textcolor{red}{1\%}        \\
    \hline
    \end{tabular}
    \caption{{\bf(a)} $m$ represents the number of data points; $n$ represents the dimension of data points. \protect\\
    {\bf(b)} We record the proportion by ``succeed rate". More specifically, for each pair $(m,n)$, we generate one data matrix $X_{m\times n}\sim N({\bf 0},diag(10))$ and $100$ random tropical triangles ${\mathcal C}:=tconv(\{{\bf u}^{(1)},{\bf u}^{(2)},{\bf u}^{(3)}\})$. {Here, for all $i=1,2,3$, we make the first coordinate of ${\bf u}^{(i)}$ as 0, and all other coordinates of ${\bf u}^{(i)}$ obey the uniform distribution on $[-9999,9999]$.} For each triangle ${\mathcal C}$, we test if the projection of a Fermat-Weber point of $X_{m\times n}$ on ${\mathcal C}$ is a Fermat-Weber point of the projection of $X_{m\times n}$ on ${\mathcal C}$.}
    \label{tab:showMyAlgoIsGood}
\end{table}

\item[]{\bf2.} Table \ref{tab:goodRateForMNchange} shows the succeed rate of Algorithm \ref{algo:MofE} or Algorithm \ref{algo:NewMethod} (recall Remark \ref{rem:SorFsimultaneously} tells that, Algorithm \ref{algo:MofE} and Algorithm \ref{algo:NewMethod} always succeed or fail simultaneously). From Table \ref{tab:showMyAlgoIsGood} and Table \ref{tab:goodRateForMNchange} we can see that, the succeed rates recorded in Table \ref{tab:goodRateForMNchange} are much higher than those in Table \ref{tab:showMyAlgoIsGood}. For instance, the lowest rate in Table \ref{tab:goodRateForMNchange} is \textcolor{red}{$34\%$}, which is still higher than the highest rate in Table \ref{tab:showMyAlgoIsGood}, and the highest rate in Table \ref{tab:goodRateForMNchange} is \textcolor{red}{$94\%$}, which is close to $100\%$.
\begin{table}[H]\scriptsize
    \centering
    \begin{tabular}{|c|c|c|c|c|}
        \hline  \diagbox{n}{succeed rate}{m}    &   30      &   60      &   90      &   120     \\
        \hline  5                               &   86\%    &   62\%    &   53\%    &   \textcolor{red}{34\%}    \\
        \hline  10                              &   82\%    &   67\%    &   60\%    &   54\%    \\
        \hline  15                              &   89\%    &   76\%    &   76\%    &   61\%    \\
        \hline  20                              &   \textcolor{red}{94\%}    &   82\%    &   76\%    &   79\%    \\
        \hline
    \end{tabular}
    \caption{{\bf(a)} $m$ represents the number of data points; $n$ represents the dimension of data points.\protect\\
    {\bf(b)} We record the proportion as ``succeed rate". More specifically, for each pair $(m,n)$, we generate $100$ data matrices $X_{m\times n}\sim N({\bf 0},diag(10))$, run Algorithm \ref{algo:MofE} or Algorithm \ref{algo:NewMethod}, and calculate the proportion of that Algorithm \ref{algo:MofE} or Algorithm \ref{algo:NewMethod} succeeds.}
    \label{tab:goodRateForMNchange}
\end{table}

\item[]{\bf3.} {We fix $m=120$, and we fix $n=20$.} Table \ref{tab:GoodRateVchanges} shows how high the succeed rate of Algorithm \ref{algo:MofE} or Algorithm \ref{algo:NewMethod} would be {when we change the data matrix $X_{120\times20}$}. In order to change $X$, we change $v$, such that $X\sim N({\bf 0},diag(v))$. We can see from Table \ref{tab:GoodRateVchanges} that, when $v$ is changing from $1$ to $800$, the succeed rate of Algorithm \ref{algo:MofE} or Algorithm \ref{algo:NewMethod} is still around \textcolor{red}{$70\%$}. Note that $v$ is the variance of each coordinate of data points, which means that, when the coordinate of data points fluctuates violently, the succeed rate of Algorithm \ref{algo:MofE} or Algorithm \ref{algo:NewMethod} is still stable.
\begin{table}[H]\scriptsize
    \centering
    \begin{tabular}{|c|c|c|c|c|c|}
        \hline             v    &   1       &   5           &   10          &   50          &   800     \\
        \hline  succeed rate    &   67\%    &   65\%      &    73\%     &   66\%      &   67\%  \\
        \hline
    \end{tabular}
    \caption{{\bf(a)} $v$ {is a real number} such that $X_{120\times20}\sim N({\bf 0},diag(v))$.\protect\\
    {\bf(b)} We record the proportion as ``succeed rate". More specifically, for each $v$, we generate $100$ random data matrices $X_{120\times20}\sim N({\bf 0},diag(v))$, run Algorithm \ref{algo:MofE} or Algorithm \ref{algo:NewMethod}, and calculate the proportion of that Algorithm \ref{algo:MofE} or Algorithm \ref{algo:NewMethod} succeeds.}
    \label{tab:GoodRateVchanges}
\end{table}

\item[]{\bf4.} Table \ref{tab:runTimeForBoth} shows the average computational time {for} Algorithm \ref{algo:MofE} and {that for} Algorithm \ref{algo:NewMethod}. From Table \ref{tab:runTimeForBoth} we can see that, Algorithm \ref{algo:MofE} and Algorithm \ref{algo:NewMethod} are both efficient. For instance, when there are $120$ data points, and the dimension of each point is $20$, the computational {timings} of Algorithm \ref{algo:MofE} and Algorithm \ref{algo:NewMethod} are still no more than $7$ minutes (\textcolor{red}{373.5734s} and \textcolor{red}{291.9031s}). In addition, in most cases, Algorithm \ref{algo:NewMethod} takes less time than Algorithm \ref{algo:MofE} does. For instance, when $m$ is $120$, and $n$ is $20$, Algorithm \ref{algo:NewMethod} takes around one and a half minutes less than Algorithm \ref{algo:MofE} does.
\begin{table}[H]\scriptsize
    \centering
    \begin{tabular}{|c|c|c|c|c|c|c|c|c|}
    \hline
    \multirow{2}*{\diagbox[height=2.4em,width=5.5em]{n}{time}{m}}&\multicolumn{2}{c|}{30}&\multicolumn{2}{c|}{60}&\multicolumn{2}{c|}{90}&\multicolumn{2}{c|}{120}\\\cline{2-9}
                                        &A1&A4&A1&A4&A1&A4&A1&A4\\
    \hline  5     &0.0549&0.0637&0.1216&0.1289&0.2066&0.2125&0.3376&0.3406\\
    \hline  10    &0.6007&0.5286&2.3137&2.1613&5.3013&4.974&9.6333&8.8836\\
    \hline  15    &3.8845&2.5153&17.5981&14.1034&42.2672&35.6299&84.7549&76.1394\\
    \hline  20    &15.7162&8.9255&96.1014&64.3878&211.1096&174.5119&\textcolor{red}{373.5734}&\textcolor{red}{291.9031}\\
    \hline
    \end{tabular}
    \caption{{\bf(a)} $m$ represents the number of data points; $n$ represents the dimension of data points.\protect\\
    {\bf(b)} We record the average computational time (in seconds) as ``time". More specifically, for each pair $(m,n)$, we run Algorithm \ref{algo:MofE} and Algorithm \ref{algo:NewMethod} for $100$ random data matrices $X_{m\times n}\sim N({\bf 0},diag(10))$, and record the average computational time {for} Algorithm \ref{algo:MofE} and {that for} Algorithm \ref{algo:NewMethod}. \protect\\
    {\bf(c)} ``A1" means the average computational time of Algorithm \ref{algo:MofE}, and ``A4" means the average computational time of Algorithm \ref{algo:NewMethod}.}
    \label{tab:runTimeForBoth}
\end{table}

\item[]{\bf5.} Figure \ref{fig:VarChangingCompare} compares the {numbers of} traversal steps of Algorithm \ref{algo:MofE} and Algorithm \ref{algo:NewMethod}. From Figure \ref{fig:VarChangingCompare} we can see that, with the proportion more than $50\%$, Algorithm \ref{algo:MofE} always takes more traversal steps than Algorithm \ref{algo:NewMethod} does.
\begin{figure}[H]
    \centering
    \subfigure{
    \centering
    \begin{minipage}[]{0.7\linewidth}
    \includegraphics[width=0.5\textwidth, height=0.5\textwidth]{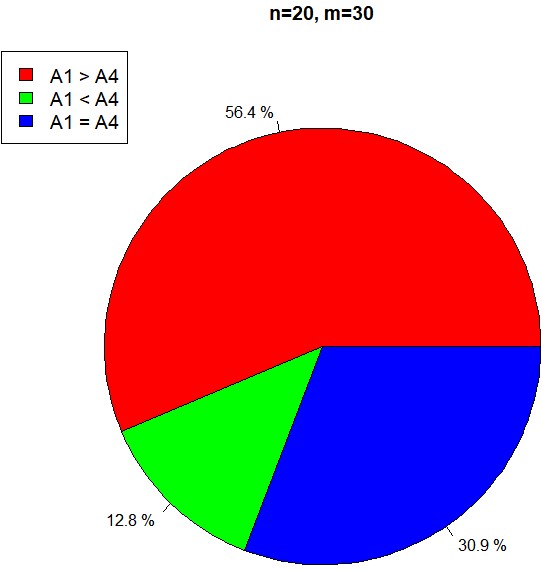}
    \end{minipage}

    \begin{minipage}[]{0.7\linewidth}
    \includegraphics[width=0.5\textwidth, height=0.5\textwidth]{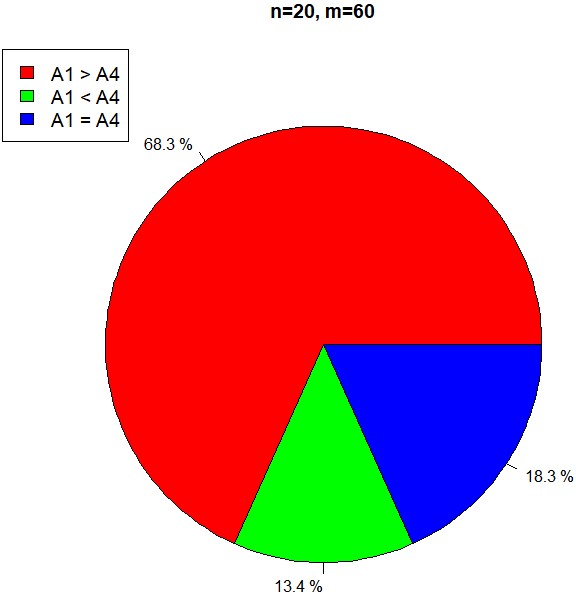}
    \end{minipage}
    }

    \subfigure{
    \centering
    \begin{minipage}[]{0.7\linewidth}
    \includegraphics[width=0.5\textwidth, height=0.5\textwidth]{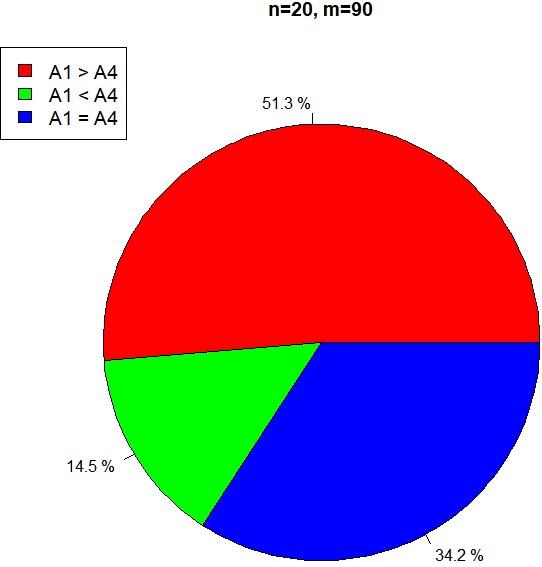}
    \end{minipage}

    \begin{minipage}[]{0.7\linewidth}
    \includegraphics[width=0.5\textwidth, height=0.5\textwidth]{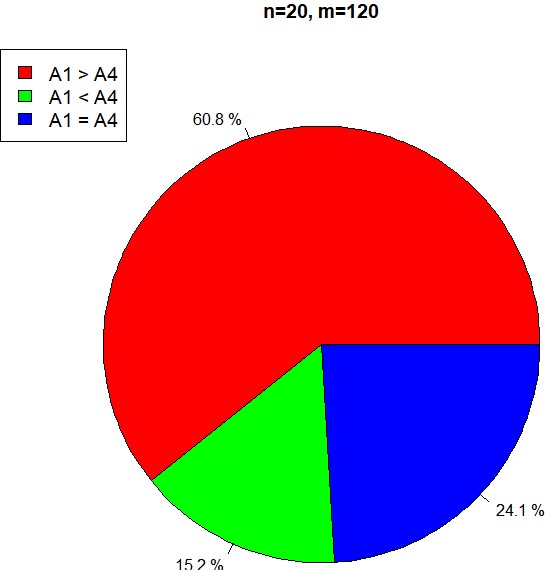}
    \end{minipage}
    }
    \caption{{\bf(a)} $m$ represents the number of data points; $n$ represents the dimension of data points.\protect\\
    {\bf(b)} ``$A1\;>\;A4$" means Algorithm \ref{algo:MofE} takes more steps than Algorithm \ref{algo:NewMethod} does.\protect\\
    {\bf(c)} ``$A1\;<\;A4$" means Algorithm \ref{algo:MofE} takes less steps than Algorithm \ref{algo:NewMethod} does.\protect\\
    {\bf(d)} ``$A1\;=\;A4$" means Algorithm \ref{algo:MofE} takes equal steps to Algorithm \ref{algo:NewMethod}.\protect\\
    {\bf(e)} For each pair $(m,n)$, we run Algorithm \ref{algo:MofE} and Algorithm \ref{algo:NewMethod} with $100$ random data matrices $X_{m\times n}\sim N({\bf 0},diag(10))$. If Algorithm \ref{algo:MofE} and Algorithm \ref{algo:NewMethod} correctly terminate, then record the number of traversal steps that Algorithm \ref{algo:MofE} and Algorithm \ref{algo:NewMethod} respectively take.}
    \label{fig:VarChangingCompare}
\end{figure}
\end{itemize}

\newpage
\bibliographystyle{plain}
%\bibliography{ref}

\end{document}